\documentclass{amsart}
\usepackage{amssymb}
\usepackage{graphicx}
\def\({\left(}
\def\){\right)}

\newtheorem{lema}{Lemma}[section]

\newtheorem*{teorema*}{Theorem}

\newtheorem{theorem}[lema]{Theorem}

\newtheorem{definition}[lema]{Definition}

  {\par \hfill \fbox{}}
  {\par \hfill \fbox{}}
  {\par \hfill }
  {\par \hfill }

\def\beq{\begin{equation}}
\def\eeq{\end{equation}}

\def\epsilon{\varepsilon}

\begin{document}

\title{Arithmetic continuity in quasi cone metric spaces}

\author{Shallu Sharma}
\address{Department of Mathematics, University of Jammu } 
\email{shallujamwal09@gmail.com}

\author{Iqbal Kour}
\address{Department of Mathematics, University of Jammu}
\email{iqbalkour208@gmail.com}


\date{}

\begin{abstract}  
 In this paper we have introduced arithmetic ff-continuity and arithmetic fb-continuity utilizing the concept of forward and backward arithmetic convergence in quasi cone metric spaces. These concepts are used to prove some fascinating results. The notions of forward and backward arithmetic compactness in quasi cone metric spaces are also established which are further utilized to prove some related results.      
\end{abstract}

\maketitle{ }

\section{Introduction}
A distance function is said to be quasi metric if symmetric condition is dropped from the definition of metric (see \cite{A,CZ,K,M,W,YH}). Various definitions of quasi cone metric have been given by various authors (see, for example, \cite{AK}). Since then, much study has been conducted on the quasi cone metric, particularly on fixed point theory. The notion of arithmetic convergence was introduced by Ruckle \cite{R} in the form of a sequence $\{\mathsf x_{n}\}$ defined on the set of natural numbers. The sequence $\{\mathsf x_{n}\}$ is said to be arithmetic convergent if for every $\mathsf u>0$ there exists $m\in \mathbb{Z}$ such that for each $n\in \mathbb{Z}, |\mathsf x_{n}-\mathsf x_{<n,m>}|<\mathsf u,$ where the greatest common divisor of $n$ and $m$ is denoted as $<n,m>.$ Another definition of arithmetic convergence given by Cakalli in \cite{C1} as : a sequence $\{\mathsf x_{n}\}$ is called arithmetically convergent if for each $u>0$ there exists $m_{0}\in \mathbb{Z}$ such that for all $n,m\in \mathbb{Z}$ that satisfy $<n,m>\geq m_{0},$ we have $|\mathsf x_{n}-\mathsf x_{<n,m>}|<u.$ For comprehensive details on arithmetic continuity and arithmetic convergence one can refer \cite{C1,YH1,YH2,YH3,YH4}. In this paper, the concepts of forward and backward arithmetic convergence in a quasi cone metric space are firstly introduced. These concepts are used to define forward and backward arithmetic continuity in quasi cone metric spaces which are further utilized to obtain some fascinating results. The notion of forward and backward arithmetic compactness is also introduced.
\section{Preliminaries}
\begin{definition}\cite{HZ}
Let $\mathsf E\subseteq \mathbb{R}$ be a Banach space. A set $\mathsf P$ contained in $\mathtt E$ is said to be a cone if
\begin{itemize}
\item[(i)] $\mathsf P$ is non-empty, closed and non-zero.
\item[(ii)] $\mathsf x, \mathsf y$ are elements of $\mathsf P$ and $\mathsf s,\mathsf t$ are positive real numbers, then $\mathsf s \mathsf x+\mathsf t \mathsf y$ is an element of $\mathsf P.$
\item[(iii)] the intersection of $\mathsf P$ and $- \mathsf P$ is $\{0\}.$ 
\end{itemize}
\end{definition}
Assuming $\mathsf P$ be a cone in $\mathsf E,$ we define the partial ordering relation $\leq$ on $\mathsf E$ in relation to $\mathsf P$ as : $\mathsf x\ll \mathsf y$ if and only if $\mathsf y-\mathsf x$ is an element of interior of $\mathsf P,\mathtt x\leq \mathtt y$ if and only if $\mathtt y-\mathtt x$ is an element of $ \mathsf{P}.$ The partial ordering $\mathtt x<\mathtt y$ indicates that $\mathtt x\leq \mathtt y,$ $\mathtt x\neq \mathtt y.$ Throughout this paper we shall assume that $\mathsf E$ is a Banach space, $\mathsf P$ is a cone with non-empty interior with the partial ordering as defined above. Also, we shall denote closure of any set $\mathsf B$ as $Cl(\mathsf B).$
\begin{definition}\cite{AK}
 Let $\mathsf X$ be a set that is non-empty and define a mapping $d:\mathsf X\times \mathsf X\to \mathsf E$ as follows:
\begin{itemize}
\item [(i)] $d(\mathsf x, \mathsf y)\geq 0$ for $\mathsf x, \mathsf y\in \mathsf X.$
\item [(ii)] $d(\mathsf x, \mathsf y)=0$ if and only if $\mathsf x=\mathsf y.$

\item [(iii)] $d(\mathsf x,\mathsf y)\leq d(\mathsf x, \mathsf z)+d(\mathsf z,\mathsf y)$ for $\mathsf x, \mathsf y, \mathsf z$ in $\mathsf X.$
\end{itemize}
Then $(\mathsf X,d)$ is said to be a quasi cone metric space.
\end{definition}
\begin{definition}\cite{SN}
Let $(\mathsf X,d)$ be a quasi cone metric space and $\{\mathsf x_{n}\}$ be a sequence in $\mathsf X.$ The sequence $\{\mathsf x_{n}\}$ is said to be forward  convergent (resp. backward convergent) to $\mathsf x_{0}$ in $\mathsf X$ if for each $\mathsf u\gg 0,\mathsf u\in \mathsf E$ there exists  $\mathsf N\in \mathbb{N}$ so that for every $\mathsf n\geq\mathsf N,$ we have $d(\mathsf x_{0},\mathsf x_{n})\ll u~(resp.~d(\mathsf x_{n},\mathsf x_{0})\ll \mathsf u).$ We denote it as $\mathsf x_{n}\stackrel{f}{\rightarrow} \mathsf x_{0}~(resp.~\mathsf x_{n}\stackrel{b}{\rightarrow} \mathsf x_{0}) .$
\end{definition}
\begin{definition}\cite{SN}
 Let $\{\mathsf x_{n}\}$ be a sequence in a quasi cone metric space $(\mathsf X,d).$ Then the sequence $\{\mathsf x_{n}\}$ is said to be forward Cauchy (resp. backward Cauchy) if for each $\mathsf u\gg 0,\mathsf u\in \mathsf E$ there exists $\mathsf N\in \mathbb{N}$ such that for each $m\geq n\geq \mathsf N,d(\mathsf x_{n},\mathsf x_{m})\ll \mathsf u~ (resp.~d(\mathsf x_{m}, \mathsf x_{n})\ll \mathsf u).$ 
\end{definition}
\begin{definition}\cite{YHC}
A cone metric space $(\mathsf X,d)$ is said to be forward sequentially compact if each sequence in $\mathsf X$ has  a forward convergent subsequence. 
\end{definition}
\section{Main results}
\begin{definition}
Let $(\mathsf X,d)$ be a quasi cone metric space and $\{\mathsf x_{n}\}$ be a sequence in $\mathsf X.$ Then $\{\mathsf x_{n}\}$ is said to be forward arithmetic convergent (resp. backward arithmetic convergent) if for every $\mathsf u\in \mathsf E,u\gg 0$ there exists $\mathsf N\in \mathbb Z$ such that $d(\mathsf x_{<n,m>},\mathsf x_{n})\ll u(resp.~ d(\mathsf x_{n},\mathsf x_{<n,m>})\ll u),$ whenever $<n,m>\geq \mathsf N,n,m\in \mathbb Z.$ It shall be denoted as $\mathsf x_{n}\stackrel{amf}{\rightarrow} \mathsf x_{<n,m>} (resp.~ \mathsf x_{n}\stackrel{amb}{\rightarrow}\mathsf x_{<n,m>}).$    
\end{definition}
\begin{definition}
Let $(\mathsf X,d_{\mathsf X})$ and $(\mathsf Y,d_{\mathsf Y})$ be two quasi cone metric spaces. A function $f$ from $\mathtt X$ to $\mathtt Y$ is said to be arithmetic $ff-$continuous (resp. arithmetic $fb-$continuous) at a point $x\in \mathsf X$ if $x_{n}\stackrel{amf}{\rightarrow} x$  in $(\mathsf X,d_{\mathsf X})$ implies that $f(x_{n})\stackrel{amf}{\rightarrow} f(x)(resp.~f(x_{n})\stackrel{amb}{\rightarrow}f(x))$ in $(\mathsf Y,d_{\mathsf Y}).$ 
\end{definition}
\begin{definition}
Let $(\mathsf X,d_{\mathsf X})$ and $(\mathsf Y,d_{\mathsf Y})$ be two quasi cone metric spaces. A function $f$ from $\mathsf X$ to $\mathsf Y$ is said to be uniformly continuous if for each $\mathsf x,\mathsf y\in \mathsf X,\mathsf u'\in \mathsf E',\mathsf u'\gg 0$ there exists $\mathsf u\in \mathsf E,\mathsf u\gg 0$ such that $d_{\mathsf Y}(f(\mathsf x),f(\mathsf y))\ll u',$ whenever $d_{\mathsf  X}(\mathsf x,\mathsf y)\ll u.$
\end{definition} 
Note that forward uniform continuity and backward uniform continuity are same.
\begin{theorem}
Let $(\mathsf X,d_{\mathsf X})$ and $(\mathsf Y,d_{\mathsf Y})$ be two quasi cone metric spaces. A function $f$ from $\mathsf X$ to $\mathsf Y$ is arithmetic $ff-$continuous if $f$ is uniformly continuous.
\end{theorem}

\begin{proof}
Suppose that $f$ is uniformly continuous and $\{\mathsf x_{n}\}$ is a sequence in $\mathsf X$ that is forward arithmetic convergent. Then for every $\mathsf u'\in \mathsf E',u'\gg 0$ there exists $\mathsf u\in \mathsf E,\mathsf u\gg 0$ such that $d_{\mathsf Y}(f(\mathsf x),f(\mathsf y))\ll \mathsf u',$ whenever $d_{\mathsf X}(\mathsf x,\mathsf y)\ll \mathsf u$ as $f$ is uniformly continuous. Also, $\{\mathsf x_{n}\}$ is forward arithmetic convergent in $\mathsf X.$ Then for above $\mathsf u\gg 0,$ there exists $n_{0}\in \mathbb N$ such that for all $m,n\in \mathbb{Z}$ that satisfy $<n,m>\geq n_{0},$ we have $d_{\mathsf Y}(f(\mathsf x_{<n,m>}),f(x_{n}))\ll \mathsf u'$ when $d_{\mathsf X}(x_{<n,m>},x_{n})\ll \mathsf u$ for each $n.$ Therefore, $\{f(\mathsf x_{n})\}$ is forward arithmetic convergent sequence. Hence $f$ is arithmetic $ff-$continuous.
\end{proof}       

\begin{definition}
Let $(\mathsf X,d_{\mathsf X})$ and $(\mathsf Y,d_{\mathsf Y})$ be two quasi cone metric spaces. A sequence $\{f_{n}\}$ from $\mathsf X$ to $\mathsf Y$ is said to be forward arithmetic convergent (resp. backward arithmetic convergent) if for every $\mathsf u'\in \mathsf E',u'\gg 0$ and for every $\mathsf x\in \mathsf X$ there exists $n_{0}\in \mathbb{N}$ such that for each $m,n\in \mathbb{Z}$ that satisfy $<m,n>\geq n_{0},$ we have $d_{\mathsf Y}(f_{<n,m>}(x),f_{n}(x))\ll u'(resp.~ d_{\mathsf Y}(f_{n}(\mathsf x),f_{<n,m>}(x))\ll u').$   
\end{definition}
\begin{theorem}\label{th:4}
Let $(\mathsf X,d_{\mathsf X})$ and $(\mathsf Y,d_{\mathsf Y})$ be two quasi cone metric spaces. If $\{f_{n}\}$ is a sequence of forward arithmetic convergent functions from $\mathsf X$ to $\mathsf Y$ and $\mathsf x_{0}\in \mathsf X$ so that $f_{n}(\mathsf x)$ tends to $\mathsf y_{n}$ as $\mathsf x$ tends to $\mathsf x_{0},$ then the sequence $\{\mathsf y_{n}\}$ is also forward arithmetic convergent.
\end{theorem}
\begin{proof}
By the definition of forward arithmetic convergence, for $u'\in \mathsf E',u'\gg 0$ there exists $n_{0}\in \mathbb{N}$ such that for every $m,n\in \mathbb{Z}$ that satisfy $<n,m>\geq n_{0}$ and for each $x\in \mathsf X,$  we have $d_{\mathsf{Y}}(f_{<n,m>}(x),f_{n}(x))\ll \mathsf u'.$ Fix $n,m$ and suppose $x$ tends to $x_{0},$ we have $d_{\mathsf Y}(\mathsf y_{<n,m>},\mathsf y_{n})\ll u'.$ Therefore, $\{\mathsf y_{n}\}$ is forward arithmetic convergent. This proves the theorem.   
\end{proof}
\begin{theorem}
Let $(\mathsf X,d_{\mathsf X})$ and $(\mathsf Y,d_{\mathsf Y})$ be two quasi cone metric spaces. If $\{f_{n}\}$ is a sequence of backward arithmetic convergent functions from $\mathsf X$ to $\mathsf Y$ and $\mathsf x_{0}\in \mathsf X$ so that $f_{n}(\mathsf x)$ tends to $\mathsf y_{n}$ as $\mathsf x$ tends to $\mathsf x_{0},$ then the sequence $\{\mathsf y_{n}\}$ is also backward arithmetic convergent.
\end{theorem}
\begin{proof}
The proof is omitted as it follows trivially from Theorem \ref{th:4}.
\end{proof}
\begin{theorem}
Let $(\mathsf X,d_{\mathsf X})$ and $(\mathsf Y,d_{\mathsf Y})$ be two quasi cone metric spaces and $\{f_{n}\}$ be a sequence of arithmetic $ff-$continuous functions from $\mathsf X$ to $\mathsf Y,$ forward convergence is equivalent to backward convergence in $\mathsf Y$ and $\{f_{n}\}$ is forward convergent to $f$ uniformly. Then $f$ is arithmetic $ff-$continuous. 
\end{theorem}
\begin{proof}
Choose $\mathsf u'\in \mathsf E',\mathsf u'\gg 0.$ Suppose that $\{\mathsf x_{n}\}$ is a forward arithmetic convergent sequence in $\mathsf X.$ Since $\{f_{n}\}$ is uniformly forward convergent to $f,$ there exists $\mathsf N\in \mathbb N$ such that $d_{\mathsf Y}(f(\mathsf x), f_{n}(\mathsf x))\ll \frac{u'}{3}$ for all $n\geq \mathsf N$ and $x\in \mathsf X.$ Particularly, $f_{n}(\mathsf x_{<n,m>})$ is forward convergent to $f(\mathsf x_{<n,m>})$ and hence $f_{n}(\mathsf x_{<n,m>})$ is backward convergent to $f(\mathsf x_{<n,m>}).$ Therefore, there exist $\mathsf M\in \mathbb{N}$ such that $d_{\mathsf{Y}}(f_{n}(\mathsf x_{<n,m>}),f(\mathsf x_{<n,m>}))\ll \frac{u'}{3}$ for all $n\geq \mathsf M.$ Let $\mathsf N'=\max. (\mathsf N,\mathsf M).$ Also, $\{f_{n}\}$ is a sequence of arithmetic $ff-$continuous functions. Particularly, $f_{\mathsf N'}$ is arithmetic $ff-$continuous function. Since forward convergence in $\mathsf Y$ is equivalent to backward convergence. Then $f_{\mathsf N'}$ is also arithmetic $fb-$continuous. Therefore, there exists $n_{0}> \mathsf{N}$ and $\mathsf u\in \mathsf E,\mathsf u\gg 0$ such that $d_{\mathsf X}(\mathsf x_{<n,m>},\mathsf x_{n})\ll u$ implies that $d_{\mathsf Y}(f_{\mathsf N'}(\mathsf x_{n}),f_{\mathsf N'}(\mathsf x_{<n,m>})\ll \frac{u'}{3}\forall m,n\in \mathbb{N}$ such that $<n,m>\geq m_{0}.$ Moreover, for $d_{\mathsf X}(x_{<n,m>},x_{n})\ll u $ and $<n,m>\geq m_{0},$ we see that 
\begin{eqnarray*}
d_{\mathsf Y}(f(\mathsf x_{n}),f(\mathsf x_{<n,m>}))
&\leq& d_{\mathsf Y}(f(\mathsf x_{n}),f_{\mathsf N'}(\mathsf x_{n}))\\
&+&d_{\mathsf Y}(f_{\mathsf N'}(\mathsf x_{n}),f_{\mathsf N'}(\mathsf x_{<n,m>}))+d_{\mathsf Y}(f_{\mathsf N'}(\mathsf x_{<n,m>}),f(\mathsf x_{<n,m>}))\\
&\ll & \frac{\mathsf u'}{3}+\frac{\mathsf u'}{3}+\frac{\mathsf u'}{3}\\
&=&\mathsf u'.
\end{eqnarray*}
Hence $f$ is arithmetic $fb-$continuous. As forward convergence is equivalent to backward convergence $f$ is also arithmetic $ff-$continuous. This proves the theorem.
\end{proof}
\begin{theorem}
Let $(\mathsf X,d_{\mathsf X})$ and $(\mathsf Y,d_{\mathsf Y})$ be two quasi cone metric spaces and forward convergence in $\mathsf Y$ is equivalent to backward convergence. The set of all arithmetic $ff-$continuous functions from $\mathsf X$ to $\mathsf Y$ is a closed subset of all continuous functions from $\mathsf X$ to $\mathsf Y.$ That is, $Cl(\mathsf A_{mff})(\mathsf X,\mathsf Y)=A_{mff}(\mathsf X,\mathsf Y).$ Here $A_{mff}(\mathsf X,\mathsf Y)$ is the set of all functions from $\mathsf X$ to $\mathsf Y$ that are arithmetic $ff-$continuous.    
\end{theorem}
\begin{proof}
Suppose $f\in Cl(\mathsf A_{mff})(\mathsf X,\mathsf Y).$ Then there is a sequence $\{f_{n}\}$ in $A_{mff}(\mathsf X,\mathsf Y)$ such that $\{f_{n}\}$ is forward convergent to $f.$ Let $\mathsf u'\in \mathsf E',\mathsf u'\gg 0.$ Consider a forward arithmetic convergent sequence $\{\mathsf x_{n}\}$ in $\mathsf {X}.$ As $f_{n}$ is forward convergent uniformly to $f.$ Then for all $\mathsf x\in \mathsf X$ there exists $\mathsf N_{1}\in \mathbb{N}$ such that $d_{\mathsf Y}(f(\mathsf x),f_{n}(\mathsf x)\ll \frac{u'}{3},\forall n\geq \mathsf {N}_{1}.$ Particularly, $f_{n}(\mathsf x_{<n,m>})$ is forward convergent to $f.$ Also, $f_{n}(\mathsf x_{<n,m>})$ is backward convergent to $f$ as forward convergence is equivalent to backward convergence. Hence there exists $\mathsf N_{2}\in \mathsf N$ such that $d_{\mathsf Y}(f_{n}(\mathsf x_{<n,m>}),f(\mathsf x_{<n,m>}))\ll \frac{u'}{3}\forall n\geq \mathsf N_{2}.$ Choose $\mathsf N=\max. \{\mathsf N_{1},\mathsf N_{2}\}.$ Also, $\{f_{n}\}$ is arithmetic $ff-$continuous. Particularly, $f_{\mathsf N}$ is arithmetic $ff-$continuous. Since forward convergence in $\mathsf Y$ is equivalent to backward convergence we see that $f_{\mathsf N}$ is also arithmetic $fb-$continuous. Hence there exists $m_{0}>\mathsf N$ and $\mathsf u\in \mathsf E,\mathsf u\gg 0$ such that $d_{\mathsf X}(\mathsf x_{<n,m>},\mathsf x_{n})\ll \mathsf u$ implies that $d_{\mathsf Y}(f_{\mathsf N}(\mathsf x_{n}),f_{\mathsf N}( \mathsf x_{<n,m>}))\ll  \frac{u'}{3},~\forall~m,n\in \mathbb{Z}$ and $<m,n>\geq m_{0}.$ Therefore, for $d_{\mathsf X}(\mathsf x_{<n,m>},\mathsf x_{n})\ll \mathsf u$ and $<m,n>\geq m_{0},$ we get
\begin{eqnarray*}
d_{\mathsf Y}( f(\mathsf x_{n}),f( \mathsf x_{<n,m>})) &\leq&d_{\mathsf Y}(f(\mathsf x_{n}),f_{\mathsf N}(\mathsf x_{n}))\\
&+&d_{\mathsf Y}(f_{\mathsf N}(\mathsf x_{n}),f_{\mathsf N}(\mathsf x_{<n,m>})))+d_{\mathsf Y}(f_{\mathsf N}(\mathsf x_{<n,m>}),f(\mathsf x_{<n,m>}))\\
&\ll& \frac{\mathsf u'}{3}+\frac{\mathsf u'}{3}+\frac{\mathsf u'}{3}\\
&=&\mathsf u'.
\end{eqnarray*}
Hence $f$ is arithmetic $fb-$continuous. Also $f$ is arithmetic $ff-$continuous as forward convergence is equivalent to backward convergence. Therefore, $\mathsf A_{mff}(\mathsf X,\mathsf Y)$ contains $f.$ Hence the proof. 
\end{proof}

\begin{definition}
Let $(\mathsf X,d_{\mathsf X})$ and $(\mathsf Y,d_{\mathsf Y})$ be two quasi-cone metric spaces. A function $f:\mathsf X\to \mathsf Y$ is said to be forward $(c\mathsf {A}\mathsf{C})-$continuous if it maps forward convergent sequences in $\mathsf X$ to forward arithmetic convergent sequences in $\mathsf Y.$ That is, $\{f(\mathsf x_{n})\}$ is a forward arithmetic convergent sequence in $\mathsf Y,$ whenever $\{\mathsf x_{n}\}$ is a forward convergent sequence in $\mathsf X.$  

\end{definition}
\begin{theorem}\label{th:1}
Let $(\mathsf X,d_{\mathsf X})$ and $(\mathsf Y,d_{\mathsf Y})$ be two quasi-cone metric spaces such that forward convergence in $\mathsf Y$ is equivalent to backward convergence and $\{f_{n}\}$ be a sequence consisting of forward $(c\mathsf A \mathsf C)-$continuous functions from $\mathsf X$ to $\mathsf Y$ such that $\{f_{n}\}$ is forward convergent to $f$ uniformly. Then $f$ is forward $(c\mathsf A \mathsf C)-$continuous. 
\end{theorem}
\begin{proof}
Let $\mathsf u'\in \mathsf E',\mathsf u'\gg 0$ and consider a sequence $\{\mathsf x_{n}\}$ in $\mathsf X$ that is forward convergent. As $\{f_{n}\}$ is forward convergent to $f$ uniformly, then there exists $\mathsf N_{1}\in \mathbb N$ such that $d_{\mathsf {Y}}(f(x),f_{n}(x))\ll \frac{\mathsf u'}{3},~\forall~\mathsf x\in \mathsf X$ and $n\geq \mathsf N_{1}.$ Particularly, $\{f_{n}(\mathsf x_{n})\}$ is forward convergent to $f.$ Since forward convergence is equivalent to backward convergence. Then $\{f_{n}(\mathsf x_{n})\}$ is backward convergent to $f.$ Hence there exists $\mathsf N_{2}\in \mathbb{N}$ such that $d_{\mathsf Y}(f_{n}(\mathsf x_{n}),f(\mathsf x_{n}))\ll \frac{u'}{3},~\forall~n\geq \mathsf{N_{2}}.$ Take $\mathsf N=\max\{\mathsf N_{1},\mathsf N_{2}\}.$ Since $f_{n}$ is forward $(c\mathsf A \mathsf C)-$continuous. Then there exists $m_{0}> \mathsf {N}$ such that $d_{\mathsf Y}(f_{\mathsf N}(\mathsf x_{<n,m>}),f_{\mathsf N}(\mathsf x_{n}))\ll \frac{\mathsf u'}{3},~\forall~\mathsf x\in \mathsf X$ and $m,n\in \mathbb{Z}$ such that $<m,n>\geq m_{0}.$ Hence we have 
\begin{eqnarray*}
& & d_{\mathsf Y}(f(\mathsf x_{<n,m>}),f(\mathsf x_{n}))\\
&\leq&d_{\mathsf Y}(f(\mathsf x_{<n,m>}),f_{\mathsf N}(\mathsf x_{<n,m>}))+d_{\mathsf Y}(f_{\mathsf N}(\mathsf x_{<n,m>}),f_{\mathsf N}(\mathsf x_{n}))+d_{\mathsf Y}(f_{\mathsf N}(\mathsf x_{n}),f(\mathsf x_{n}))\\
&\ll&\frac{\mathsf u'}{3}+\frac{\mathsf u'}{3}+\frac{\mathsf u'}{3}\\
&=&\mathsf u'.
\end{eqnarray*} 
Hence the proof.
\end{proof}
\begin{theorem}
Let $(\mathsf X,d_{\mathsf X})$ and $(\mathsf Y,d_{\mathsf Y})$ be two quasi-cone metric spaces such that forward and backward convergence in $\mathsf Y$ are equivalent. Then the set of all forward $(c\mathsf A\mathsf C)-$continuous functions from $\mathsf X$ to $\mathsf Y$ is a closed subset of set of all functions from $\mathsf X$ to $\mathsf Y$ that are continuous. 
\end{theorem}
\begin{proof}
The proof is omitted as it follows trivially from Theorem \ref{th:1}.
\end{proof}
\begin{theorem}
Composition of two forward arithmetic $ff$-continuous functions in a quasi cone metric space $(\mathsf X,d)$ is arithmetic $ff$-continuous. 
\end{theorem}
\begin{proof}
Let $g,h$ be two forward arithmetic $ff$-continuous functions. We shall show that $g\circ h$ is again an arithmetic $ff$-continuous function. For this, let $\{\mathsf x_{n}\}$ be a forward  arithmetic convergent sequence in $\mathsf X.$ Since $h$ is arithmetic continuous. Then $\{h(\mathsf x_{n})\}$ is forward arithmetic convergent. Moreover, $g$ is also arithmetic $ff$-continuous. Then $\{g(h(\mathsf x_{n}))\}$ is also forward arithmetic convergent. Hence $g\circ h$ is also arithmetic $ff$-continuous.  This proves the theorem.
\end{proof}  
\begin{definition}
Let $(\mathsf X,d)$ be a quasi cone metric space and $\mathsf B\subset \mathsf X.$ Then $\mathsf B$ is said to be forward (resp. backward) arithmetic compact if every sequence in $\mathsf B$ admits forward (resp. backward) arithmetic convergent subsequence. 
\end{definition}
\begin{theorem}\label{th:2}
Let $(\mathsf X,d_{\mathsf X})$ and $(\mathsf Y,d_{\mathsf Y})$ be two quasi-cone metric spaces, $\mathsf B\subset \mathsf X$ be forward arithmetic compact and $f:\mathsf X\to \mathsf Y$ be an arithmetic $ff-$continuous function. Then $f(\mathsf B)$ is also forward arithmetic compact.
\end{theorem}
\begin{proof}
Let $\{\mathsf y_{n}\}$ be a sequence in $f(\mathsf B).$ Then $\mathsf y_{n}=f(\mathsf x_{n})$ for some $\mathsf x_{n}\in \mathsf X$ and $n\in \mathbb {N}.$ Also, $\{\mathsf x_{n}\}$ has a forward arithmetic convergent subsequence $\{\mathsf x_{n_{l}}\}.$ By the arithmetic $ff-$continuity of $f$ we see that $f(\mathsf x_{n})$ has a forward arithmetic convergent subsequence $f(\mathsf x_{n_{l}}).$ This proves that $f(\mathsf B)$ is forward arithmetic compact. 
\end{proof}
\begin{theorem}
Let $(\mathsf X,d_{\mathsf X})$ and $(\mathsf Y,d_{\mathsf Y})$ be two quasi-cone metric spaces, $\mathsf B\subset \mathsf X$ be backward arithmetic compact and $f:\mathsf X\to \mathsf Y$ be an arithmetic $fb-$continuous function. Then $f(\mathsf B)$ is also backward arithmetic compact.
\end{theorem}
\begin{proof}
The proof is omitted as it follows trivially from Theorem \ref{th:2}.
\end{proof}
\begin{theorem}\label{th:3}
Let $(\mathsf X,d)$ be a quasi cone metric space and $\mathsf B\subset \mathsf X$ be forward arithmetic compact. Then any closed subset of $\mathsf B$ is forward arithmetic compact.
\end{theorem}
\begin{proof}
Let $\mathsf B\subset \mathsf X$ be forward arithmetic compact and $\mathsf A$ be a closed set contained in $\mathsf B.$ Consider a sequence $\{\mathsf x_{n}\}$ in $\mathsf A.$ Then $\{\mathsf x_{n}\}$ is a sequence in $\mathsf B.$ Also, the sequence $\{\mathsf x_{n}\}$ has a forward arithmetic convergent subsequence $\{\mathsf x_{n_{l}}\}$ as $\mathsf B$ is forward arithmetic compact. Moreover, $\mathsf A$ is closed. Then the sequence $\{\mathsf x_{n}\}$ in $\mathsf A$ has a forward arithmetic convergent subsequence in $\mathsf A.$ Hence the proof. 
\end{proof}
\begin{theorem}
Let $(\mathsf X,d)$ be a quasi cone metric space and $\mathsf B\subset \mathsf X$ be backward arithmetic compact. Then any closed subset of $\mathsf B$ is backward arithmetic compact.
\end{theorem}
\begin{proof}
The proof is omitted as it follows trivially from Theorem \ref{th:3}.

\end{proof}

\end{document}